\documentclass[12pt]{amsart}

\usepackage[margin=1in]{geometry}
\usepackage{amsmath,amssymb,amsthm,xcolor}
\usepackage{graphicx}
\usepackage{caption}
\usepackage{subcaption}
\usepackage{tikz, url}
\usetikzlibrary{shapes,snakes}
\usetikzlibrary{arrows}
\usetikzlibrary{shapes.misc}
\usetikzlibrary{decorations.markings}
\usepackage{tikz-cd}

\usepackage{wrapfig}

\newtheorem{theorem}{Theorem}
\newtheorem{lemma}[theorem]{Lemma}

\numberwithin{equation}{section}
\numberwithin{theorem}{section}

\newcommand{\RR}{\mathbb{R}}
\newcommand{\QQ}{\mathbb{Q}}

\newcommand{\calH}{\mathcal{H}}
\newcommand{\ZZ}{\mathbb{Z}}
\newcommand{\vol}{\mathrm{vol}}
\newcommand{\calC}{\mathcal{C}}
\newcommand{\NN}{\mathbb{N}}
\newcommand{\CC}{\mathbb{C}}

\renewcommand\Re{\textnormal{Re}}

\newcommand{\SL}{\mathrm{SL}}
\renewcommand{\pmod}[1]{\,\left(\mathrm{mod}\,#1\right)}

\title{Asymptotic identities for additive convolutions of sums of divisors}

\author{Robert J. Lemke Oliver}
\address{Department of Mathematics, Tufts University, 503 Boston Ave, Medford, MA 02155}
\email{robert.lemke\_{}oliver@tufts.edu}
\author{Sunrose T. Shrestha}
\address{Department of Mathematics and Computer Science, Wesleyan University, 45 Wyllys Ave, Middletown, CT 06459}
\email{sunrose.shrestha@gmail.com}
\author{Frank Thorne}
\address{Department of Mathematics, University of South Carolina, 1523 Greene St, Columbia, SC 29201}
\email{thorne@math.sc.edu}

\begin{document}

\begin{abstract}
In a 1916 paper, Ramanujan studied the additive convolution $S_{a, b}(n)$ of sum-of-divisors functions $\sigma_a(n)$ and $\sigma_b(n)$, and proved an asymptotic
formula for it when $a$ and $b$ are positive odd integers. He also conjectured that his asymptotic formula should hold for all positive 
real $a$ and $b$. Ramanujan's conjecture was subsequently proved by Ingham, and then by Halberstam with a power saving error term.

%~ In this paper we give two proofs of Ramanujan's conjecture. The first is quite simple and is largely equivalent to Halberstam's.
%~ The second
In this paper, we give a new proof of Ramanujan's conjecture that obtains lower order terms in the asymptotics for most ranges of the parameters.  We also describe a connection to a counting problem in geometric topology that was studied in the second author's thesis and which
served as our initial motivation in studying this sum.

\end{abstract}

\maketitle

\section{Introduction}

For any integer $a$, let $\sigma_a(n)$ denote the sum of the $a$th powers of the divisors of $n$, that is,
\[
\sigma_a(n) = \sum_{d\mid n} d^a.
\]
While the particular value of $\sigma_a(n)$ depends crucially on the divisibility properties of $n$, there are nevertheless many beautiful identities dating back to a 1916 paper of Ramanujan \cite{Ramanujan} relating additive convolutions of some of these functions to others.  For positive integers $a$ and $b$, let
\[
S_{a,b}(n) := \sum_{k=1}^{n-1} \sigma_a(k) \sigma_b(n-k).
\]
Perhaps the most well-known identity is
\[
S_{3,3}(n) = \frac{1}{120}\sigma_7(n) + \frac{1}{120}\sigma_3(n)
\]
but Ramanujan establishes eight other exact identities of this type.  He also establishes the asymptotic identity
\begin{equation} \label{eqn:ramanujan}
S_{a,b}(n) = \frac{\Gamma(a+1)\Gamma(b+1)}{\Gamma(a+b+2)} \frac{\zeta(a+1)\zeta(b+1)}{\zeta(a+b+2)} \sigma_{a+b+1}(n) - \frac{1}{2}\zeta(-a)\sigma_b(n)+ O(n^{\frac{2}{3}(a+b+1)}) 
\end{equation}
for all positive odd integers $b\geq a>1$; there is an analogous formula with an additional lower order term if either $a$ or $b$ is equal to $1$.  The error term in the above asymptotic is related to the Fourier coefficients of holomorphic modular forms on $\mathrm{SL}_2(\mathbb{Z})$, and today, Ramanujan's paper is most famous for being the origin of the celebrated Ramanujan conjectures on the properties and size of these coefficients.

At the top of the second page of his paper, however, Ramanujan remarks, ``It seems very likely that (the main part of the asymptotic in \eqref{eqn:ramanujan}) is true for all positive (real) values of $a$ and $b$, but this I am at present unable to prove.''  This less well known conjecture of Ramanujan was established in 1927 by Ingham \cite{Ingham}, and then with a power saving error term 
in 1957 by Halberstam \cite{halberstam}.   Halberstam later \cite{halberstam2} proved that if both parameters are small, in that they satisfy $a+b<1$, then there is a secondary term given by a different expression in this asymptotic formula.  This formula does not, however, recover the secondary term in Ramanujan's formula \eqref{eqn:ramanujan}, both owing to its different formulation and to the requirement that $a+b<1$.

\medskip 
In this paper we give another proof of the asymptotic in \eqref{eqn:ramanujan}, improving upon the result
by establishing lower-order terms in the asymptotic for many ranges of the parameters that recover Ramanujan's secondary term. We begin with the following theorem on what is typically the largest of these lower order terms.

\begin{theorem} \label{thm:st1}
	If $a$ and $b$ are positive real numbers with $b>a \geq 1$, then
	\begin{align*}
		S_{a,b}(n) 
			&= \frac{\Gamma(a+1)\Gamma(b+1)}{\Gamma(a+b+2)}\frac{\zeta(a+1)\zeta(b+1)}{\zeta(a+b+2)}\sigma_{a+b+1}(n)  + \frac{\zeta(1-a)\zeta(b+1)}{(b+1)\zeta(b-a+2)}n^a\sigma_{b-a+1}(n) \\
			& \quad\quad+ O(n^b) + O(n^{\frac{a+b}{2}+1+\epsilon}).
	\end{align*}
\end{theorem}

Notice that when $a$ is an odd integer $\geq 3$, the secondary term in Theorem \ref{thm:st1}, which is $O(n^{b+1})$, actually vanishes, so Theorem \ref{thm:st1} is consistent with \eqref{eqn:ramanujan} (which requires both parameters to be odd integers) but does not quite recover it.  In fact, our proof shows that there are typically \emph{many} lower order terms in the asymptotic formula for $S_{a,b}(n)$, of orders $O(n^{b+1-m})$ for non-negative integers $0 \leq m < \frac{b-a}{2}+\frac{7}{4}$.  All of these terms but that of order $O(n^b)$ vanish if the smaller parameter $a$ is an odd integer, and it is in fact this term that recovers Ramanujan's secondary term.

%~ In fact, beyond the secondary term present in Theorem~\ref{thm:st1}, which is on the order $O(n^{b+1})$, our second proof shows that there are terms on the order $O(n^{b+m})$ for each integer $m$ satisfying $-\frac{b-a}{2}+\frac{3}{4} < m \leq 0$. It also reveals that the structure of these terms is different depending on whether the smaller parameter is an odd integer, regardless of the properties of the larger parameter.

\begin{theorem}\label{thm:st2}
	Let $a$ and $b$ be positive real numbers.  If $b-a > 3/2$, then
	\begin{align*}
		S_{a,b}(n) 
			&= \frac{\Gamma(a+1)\Gamma(b+1)}{\Gamma(a+b+2)}\frac{\zeta(a+1)\zeta(b+1)}{\zeta(a+b+2)}\sigma_{a+b+1}(n) \\ 
			\notag &\quad  + \frac{\zeta(1-a)\zeta(b+1)}{(b+1)\zeta(b-a+2)}n^a\sigma_{b-a+1}(n) 
			 + \sum_{0 \leq m < \frac{b-a}{2}-\frac{3}{4}} \mathrm{Res}(-m) + O_{a,b,\epsilon}(n^{\frac{a+b}{2}+\frac{3}{4}+\epsilon}),
	\end{align*}
	where $\mathrm{Res}(-m)$ is given explicitly by \eqref{eqn:s=-m}.  It satisfies $\mathrm{Res}(-m) \ll n^{b-m}$ in general, and if $a$ is an odd integer, then $\mathrm{Res}(0) = -\frac{1}{2}\zeta(-a)\sigma_b(n)$ and $\mathrm{Res}(-m)=0$ for each $m \geq 1$.
\end{theorem}

In particular, when $a\geq 3$ is an odd integer and $b > a+3/2$, Theorem \ref{thm:st2} implies
	\[
		S_{a,b}(n) 
			= \frac{\Gamma(a+1)\Gamma(b+1)}{\Gamma(a+b+2)}\frac{\zeta(a+1)\zeta(b+1)}{\zeta(a+b+2)}\sigma_{a+b+1}(n)
			-\frac{1}{2}\zeta(-a)\sigma_b(n)
			+O_{a,b,\epsilon}(n^{\frac{a+b}{2}+\frac{3}{4}+\epsilon}),
	\]
recovering Ramanujan's formula \eqref{eqn:ramanujan} but without requiring $b$ to be an odd integer.  Thus, Theorem \ref{thm:st2} recovers and expands on the asymptotic formula for $S_{a,b}(n)$ available from the theory of modular forms.  We note that when $b$ is also an odd integer, it was conjectured by Ramanujan and proved by Deligne that the error term is of the form $O_{a,b,\epsilon}(n^{\frac{a+b}{2}+\frac{1}{2}+\epsilon})$.  This improved error term is available only when $b$ is an odd integer, however; we discuss possible improvements to the error term when $b$ is not an odd integer in the final section of this paper.

%~ Our work is complementary to work of Halberstam \cite{halberstam2} that obtains lower order terms in Ramanujan's conjecture when $a + b < 1$ by means of the circle method.
\medskip 

The core of the paper is Section \ref{sec:st}, where we state and prove a theorem subsuming Theorems \ref{thm:st1} and \ref{thm:st2}. 
We first present in Section \ref{sec:elem} a simple elementary proof of Ramanujan's conjecture (with power saving error term) along similar lines as Halberstam \cite{halberstam}.

Also in this paper, in Section \ref{sec:ggt} we describe a problem in geometric topology which initially motivated our interest in this problem. In brief, the additive convolution $S_{1,2}(n)$ appears while counting primitive ramified degree $n$ covers of the square torus (or in other words, square-tiled surfaces with $n$ squares) with two ramification points. These surfaces can be classified according to their horizontal cylinder configurations.  There are exactly four such configurations, and knowing the asymptotic for $S_{1,2}(n)$, which already is difficult to find in the literature, enables us to compute asymptotic proportions of two of these four horizontal cylinder configurations.

\section*{Acknowledgements}

The authors would like to thank Bruce Berndt, Michael Filaseta, Peter Humphries, Karl Mahlburg, Ken Ono, Ian Petrow, Igor Shparlinski, and Matt Young for useful discussions
and for pointing us to relevant related works.

RJLO was partially supported by NSF grant DMS-1601398. FT was partially supported by grants from the Simons Foundation (Nos. 563234 and 586594).

\section{Motivation from Geometric Topology}\label{sec:ggt}

Our initial interest in studying additive convolutions of the kind $S_{a,b}$ arose from a counting problem in geometric topology. 
In order to describe succinctly where the additive convolution appears we begin with a brief exposition on translation surfaces and their moduli spaces.

\subsection{Translation surfaces and their moduli spaces}

A \emph{translation surface} is a closed orientable surface obtained from the union of finitely many Euclidean polygons $\{\Delta_1, \dots, \Delta_n\}$ such that:
\begin{itemize}
\item the embedding of the polygons in $\RR^2$ is fixed only up to translation;
\item the boundary of every polygon is oriented counterclockwise; and
\item for every $1 \leq j \leq n$ and for every oriented side $s_j$ of $\Delta_j$, there exist $1 \leq \ell \leq n$ and an oriented side $s_\ell$ of $\Delta_\ell$ so that $s_j$ and $s_\ell$ are parallel, of equal length and of opposite orientation. The sides $s_j$ and $s_\ell$ are glued together by a parallel translation. 
\end{itemize}

A few key things follow from the definition.
\begin{itemize}
\item The total angle around a vertex is $2 \pi (k+1)$ for some non-negative integer $k$. When $k > 0$, we call the point a \emph{cone point}. 
\item We distinguish between two polygons one obtained from the other by a nontrivial rotation. However, two polygons are ``cut, parallel transport, and paste'' equivalent. For instance, consider Figure \ref{fig:equivalence}. Hence, translation surfaces come with a well defined vertical direction.
\end{itemize}

\begin{figure}[h!!]
\begin{tabular}{ccc}
\begin{tikzpicture}[sq/.style=
  {shape=regular polygon, regular polygon sides=4, draw, minimum width=2.828cm}]
\draw (0,0) node [sq]{};
\draw (0,0.8) -- (0,1.2);
\draw (0,-0.8) -- (0,-1.2);
\draw (0.8, 0.05) -- (1.2, 0.05);
\draw (0.8, -0.05) -- (1.2, -0.05);

\draw (-0.8, 0.05) -- (-1.2, 0.05);
\draw (-0.8, -0.05) -- (-1.2, -0.05);

\draw (2,0) node {$\not \simeq$};

\draw (4,0) node [sq, rotate = 35] {};
\begin{scope}[shift = {(4, 0)}, rotate = 35]
\draw (0,0.8) -- (0,1.2);
\draw (0,-0.8) -- (0,-1.2);
\draw (0.8, 0.05) -- (1.2, 0.05);
\draw (0.8, -0.05) -- (1.2, -0.05);

\draw (-0.8, 0.05) -- (-1.2, 0.05);
\draw (-0.8, -0.05) -- (-1.2, -0.05);
\end{scope}

\end{tikzpicture} &\hspace{0.5cm} & 
\begin{tikzpicture}[sq/.style=
  {shape=regular polygon, regular polygon sides=4, draw, minimum width=2.828cm}]

\draw (0,0) node [sq]{};
\draw (0,0.8) -- (0,1.2);
\draw (0,-0.8) -- (0,-1.2);
\draw (0.8, 0.05) -- (1.2, 0.05);
\draw (0.8, -0.05) -- (1.2, -0.05);

\draw (-0.8, 0.05) -- (-1.2, 0.05);
\draw (-0.8, -0.05) -- (-1.2, -0.05);

\draw (2,0) node {$\simeq$};

\begin{scope}[shift = {(3,-1)}]
\draw (0,0) -- (2,0) -- (4,2) -- (2,2) -- cycle;

\draw[dashed] (2,0)-- (2,2);

\end{scope}

\begin{scope}[shift = {(4,0)}]
\draw (0,-0.8) -- (0,-1.2);
\draw (2,0.8) -- (2,1.2);

\draw (1+0.8, 0.05) -- (1+1.2, 0.05);
\draw (1+0.8, -0.05) -- (1+1.2, -0.05);

\draw (1+-0.8, 0.05) -- (1+-1.2, 0.05);
\draw (1+-0.8, -0.05) -- (1+-1.2, -0.05);

\end{scope}

\end{tikzpicture}
\end{tabular}
\caption{On the left, the two translation surfaces differ by a nontrivial rotation, so are not considered equivalent. On the right, the two translation surfaces are cut and paste equivalent. We omit the orientation on the edges mentioned in the definition while representing the surfaces using polygons.}
\label{fig:equivalence}
\end{figure}
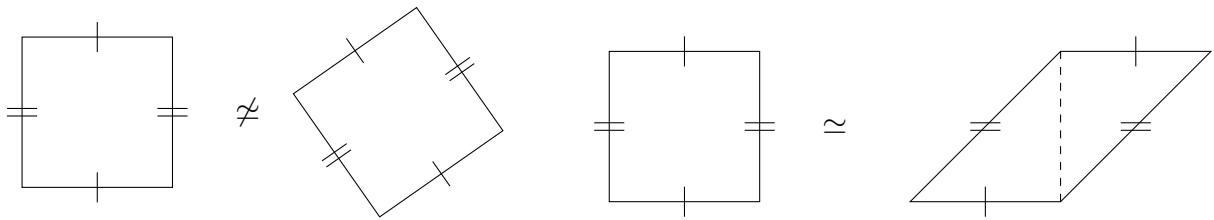

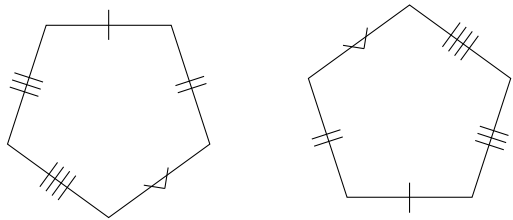
\begin{wrapfigure}{R}{0.45\textwidth}
\begin{tikzpicture}[pent/.style=
  {shape=regular polygon, regular polygon sides=5, draw, minimum width=2.828cm}]
  \draw (4,0) node[pent]{};
  \draw (0,0) node[pent,rotate=180/5]{};

  %single identification
  
  \begin{scope}[shift = {(90:1.15)}, rotate = 90]
  \draw (-0.2,0) -- (0.2,0);
  \end{scope}

  % double id
  
  \begin{scope}[shift = {(90-360/5: 1.15)}, rotate = 90-360/5]
  \draw (-0.2,0.05) -- (0.2,0.05);
   \draw (-0.2,-0.05) -- (0.2,-0.05);
  \end{scope}

   triple id
  \begin{scope}[shift = {(90+360/5: 1.15)}, rotate = 90+360/5]
  \draw (-0.2,0.1) -- (0.2,0.1);
  \draw (-0.2,0) -- (0.2,0);
  
   \draw (-0.2,-0.1) -- (0.2,-0.1);
  \end{scope}

   quadruple id
  \begin{scope}[shift = {(90+2*360/5: 1.15)}, rotate = 90+2*360/5]
  \draw (-0.2,0.15) -- (0.2,0.15);
  \draw (-0.2,0.05) -- (0.2,0.05);
   \draw (-0.2,-0.05) -- (0.2,-0.05);
  
   \draw (-0.2,-0.15) -- (0.2,-0.15);
  \end{scope}

  % v id
  \begin{scope}[shift = {(90+3*360/5: 1.07)}, rotate = 90+90+3*360/5]
  \draw (0,-0.2) -- (0.2,0);
   \draw (0,-0.2) -- (-0.2,0);
  \end{scope}
  
  %%%%% on the left:

\begin{scope}[shift = {(4,0)}] 
 
  \begin{scope}[shift = {(-90:1.15)}, rotate = 90]
  \draw (-0.2,0) -- (0.2,0);
  \end{scope}
  
  \begin{scope}[shift = {(-90-360/5: 1.15)}, rotate = -90-360/5]
  \draw (-0.2,0.05) -- (0.2,0.05);
   \draw (-0.2,-0.05) -- (0.2,-0.05);
  \end{scope}
  
  \begin{scope}[shift = {(180+90+360/5: 1.15)}, rotate = 180+90+360/5]
  \draw (-0.2,0.1) -- (0.2,0.1);
  \draw (-0.2,0) -- (0.2,0);
  
   \draw (-0.2,-0.1) -- (0.2,-0.1);
  \end{scope}
  
  % quadruple id
  \begin{scope}[shift = {(180+90+2*360/5: 1.15)}, rotate = 180+90+2*360/5]
  \draw (-0.2,0.15) -- (0.2,0.15);
  \draw (-0.2,0.05) -- (0.2,0.05);
   \draw (-0.2,-0.05) -- (0.2,-0.05);
  
   \draw (-0.2,-0.15) -- (0.2,-0.15);
  \end{scope}
  
  \begin{scope}[shift = {(180+90+3*360/5: 1.23)}, rotate = 90+90+3*360/5]
  \draw (0,-0.2) -- (0.2,0);
   \draw (0,-0.2) -- (-0.2,0);
  \end{scope}

  \end{scope}

\end{tikzpicture}
\caption{A translation surface formed by two pentagons whose opposite corresponding sides are glued. This surface has genus 2 and lives in the stratum $\calH(1,1)$.}
\vspace{-0.8cm}
\label{fig:doublepentagon}
\end{wrapfigure}

Some basic examples of translation surfaces include an axis parallel square with opposite sides identified to give a square torus and a regular octagon with opposite sides identified. One can also take two regular $n$-gons with $n$ odd and identify opposite corresponding sides to form a translation surface. Consider Figure \ref{fig:doublepentagon} for an example with $n=5$. In general, the polygons need not be regular.

Translation surfaces also admit an alternate definition via complex analysis. Viewing the polygons as embedded in $\CC$, a translation surface has a complex structure with transition functions given by translations. The globally defined 1-form $dz$ on $\CC$ then induces a globally defined 1-form $\omega$ with zeroes exactly at the cone points. Hence, from the polygonal definition of a translation surface we obtain a pair $(X, \omega)$ where $X$ is a Riemann surface and $\omega$ is holomorphic 1-form. On the other hand, given such a pair $(X, \omega)$ one can also recover the polygonal definition using a geodesic triangulation of $X$ satisfying the appropriate properties outlined in the polygonal definition. Therefore, a translation surface can also be thought of as a pair $(X, \omega)$ of a Riemann surface $X$ equipped with a holomorphic 1-form $\omega$. See \cite{Masur2} for a more precise formulation of the equivalence of these two definitions of translation surfaces.

The genus of a translation surface is given by the classical Gauss-Bonnet theorem which relates the Euler characteristic of a surface with the total curvature. Since translation surfaces are built out of Euclidean polygons, they are flat everywhere except the cone points, and the Gauss-Bonnet theorem takes on a simpler form. Hence, a surface of genus $g$ with $m$ cone points of angles $2 \pi(\alpha_1+1), \dots, 2 \pi (\alpha_m +1)$ satisfies the relation $\displaystyle 2g - 2 = \sum_{i=1}^m \alpha_i.$

The angle data around 
the cone points can be recorded in a vector $\alpha = (\alpha_1, \dots, \alpha_m)$ where $m$ is the number of cone points and $2\pi(\alpha_i+1)$ are the cone angles defined as above. The collection of translation surfaces sharing the same angle data is called a \emph{stratum} and is denoted $\calH(\alpha)$. 

For any $\alpha$ that is an integer partition of an even number, $\calH(\alpha)$ can be given the structure of a complex orbifold. The main idea is that given $(X, \omega) \in \calH(\alpha_1, \dots, \alpha_m)$, we can fix a basis $\rho_1, \dots, \rho_{2g+m-1}$ for the first homology $H_1(X, \{P_1, \dots, P_m\};\ZZ)$ relative to the cone points. We can then get a map
\begin{equation}\label{eq:periodcoords}
 \calH(\alpha) \rightarrow \CC^{2g+m-1} 
 \text{ given by } (X, \omega) \rightarrow  \bigg(\int_{\rho_1} \omega, \dots, \int_{\rho_{2g+m-1}} \omega\bigg)
\end{equation}
These are called $\emph{period coordinates}$ for $\calH(\alpha)$. The period coordinates serve as local coordinates via which it can be shown, as in \cite{Masur, Veech1, Veech2}, that the strata are complex orbifolds of dimension $2g+m-1$
where $g$ is the genus of the translation surface with cone point data $(\alpha_1, \dots, \alpha_m)$. Kontsevich and Zorich \cite{KontZor} classified the connected components of $\calH(\alpha)$ for all $\alpha$. In particular, any $\calH(\alpha)$ can have at most 3 connected components. Moreover, any stratum admits an $\SL_2(\RR)$ action --- given a translation surface built out of polygons $\{\Delta_i\}$, its image under $A \in \SL_2(\RR)$ is simply the translation surface $\{A \cdot \Delta_i\}$ where $A$ acts on the polygons linearly.

\subsection{Volume in $\calH(\alpha)$}

The period coordinates can also be used to define a volume form on $\calH(\alpha)$. Consider the linear volume form on $\CC^{2g+m-1}$, normalized so that the fundamental domain of the integer lattice $(\ZZ+i\ZZ)^{2g+m-1}$ has volume 1. The pullback of this volume form under the period map gives what is popularly called the Masur-Veech volume form on $\calH(\alpha)$. Furthermore, this induces a volume form on $\calH_1(\alpha)$, the set of translation surfaces in $\calH(\alpha)$ of area 1 (i.e. collections of surfaces with total Euclidean area of the polgyons 1). The measure of $\calH_1(\alpha)$ with respect to this induced volume form has been shown to be finite for any $\alpha$, independently by Masur \cite{Masur} and Veech \cite{Veech1}.

Twenty years after, Eskin and Okounkov \cite{EskOk} computed the volume of these strata, $\calH_1(\alpha)$. They counted a particular type of translation surfaces called \emph{square-tiled surfaces} (STSs), which are exactly those translation surfaces in which the polygons are axis parallel Euclidean unit squares. Alternatively, they are exactly those translation surfaces $(X , \omega)$ such that their image under the period map (\ref{eq:periodcoords}) is in $(\ZZ + i \ZZ)^{2g+m-1}$. In this manner, STSs have a lattice-like structure in the space of translation surfaces and can be thought of as ``integer points'' of strata. Topologically, STSs are also thought of as branched covers of the standard square-torus with branching over exactly one point. 

The idea of the volume computation is motivated by the following simple case. To compute the surface area of a body in $\RR^n$, one can consider a large dilate of the body by $R > 1$, and count the integer points inside. Asymptotically, the number of such integer points would be $c\cdot R^n$ since $\RR^n$ is $n$-dimensional. The surface area of the body is then given by 
$$ \frac{d (c\cdot R^n)}{dR}\bigg|_{R = 1} = cn.$$

To compute the volume of $\calH_1(\alpha)$, one applies the same technique. Applying a homothety to the codimension 1 subset $\calH_1(\alpha)$ by $n$, we get the set of translation surfaces surfaces of area $n$. The integer points within this dilated region in $\calH(\alpha)$ are STSs with at most $n$ squares. The asymptotics of this count then yields the volume of $\calH_1(\alpha)$.

\subsection{Connections to Number Theory}

Using the volume computation heuristic described above, Zorich \cite{Zor} computed the volume of the first few strata by hands-on counting and obtained 
$$\vol(\calH_1(\emptyset)) = 2\cdot\zeta(2); \qquad \vol(\calH_1(2)) = \frac{3}{4}\cdot \zeta(4); \qquad \vol(\calH_1(1,1)) = \frac{1}{3}\cdot\zeta(4)$$
In general, Eskin and Okounkov \cite{EskOk} showed that the volume of $\calH_1(\alpha)$ is given by
$$ 
\vol(\calH_1(\alpha)) =  \frac{(|\alpha|+1) \lim_{D \rightarrow \infty} D^{-|\alpha|-1} \sum_{d =1}^D \calC_d(\alpha)}{\dim \calH(\alpha)},
$$
where $|\alpha| = \sum \alpha_i$, and the $\calC_d$ are the coefficients of a certain generating function $ \calC(\alpha) = \sum_{d =1}^\infty \calC_d(\alpha) q^d $ which they proved to be a quasimodular form, i.e, a polynomial in the Eisenstein series $G_k(q)$ for $k = 2, 4, 6$. Consequently, they showed that
$$ \frac{\vol(\calH_1(\alpha))}{\pi^{2g}} \in \QQ$$ for any stratum $\calH(\alpha)$ of genus $g$ translation surfaces. 

Since Eskin and Okounkov's volume computations, various counting problems have received much attention in the study of STSs, including the enumeration of \emph{primitive square-tiled surfaces}, i.e. those STSs whose covering of the square torus does not factor through another STS.  In some ways this problem is analogous to counting primitive vectors in $\ZZ^n$.

In 2006, Hubert and Lelievre \cite{HubLel} and McMullen \cite{McMullen} proved that primitive $n$-square STSs in $\calH(2)$ partition into at most two orbits under the linear action of $\SL_2(\ZZ)$ (induced by the linear action of $\SL_2(\RR)$). Subsequently, Lelievre and Royer \cite{LelRoy} obtained orbit-wise counting of primitive $n$-square STSs for odd $n$ in $\calH(2)$. In the computation, they obtained and used closed forms of sums of the type
$$ 
S_{1,1}^k(n) = \sum_{\substack{(a,b) \in \NN^2 \\ ka +b = n}} \sigma_1(a)\sigma_1(b).
$$
Note that $S_{1,1}^1 = S_{1,1}$ as defined above, the convolution of $\sigma_1$ with itself. For $k=2, 4$ and $n\geq 1$,  they obtained
\begin{gather*}
S_{1,1}^2(n) = \frac{1}{12}\sigma_3(n) + \frac{1}{3}\sigma_3\left(\frac{n}{2}\right) - \frac{1}{8}n \sigma_1(n) - \frac{1}{4}n\sigma_1\left(\frac{n}{2}\right) + \frac{1}{24}\sigma_1(n) + \frac{1}{24}\sigma_1\left(\frac{n}{2}\right),\\
S_{1,1}^4(n) = \frac{1}{48}\sigma_3(n) + \frac{1}{16}\sigma_3\left(\frac{n}{2}\right) + \frac{1}{3}\sigma_3\left(\frac{n}{4}\right) - \frac{1}{16}n \sigma_1(n) - \frac{1}{4}n\sigma_1\left(\frac{n}{4}\right) + \frac{1}{24}\sigma_1(n) + \frac{1}{24}\sigma_1\left(\frac{n}{4}\right).
\end{gather*}
They were able to express these sums as linear combinations of sums of powers of divisors using the fact that the spaces of quasimodular forms on congruence subgroups such as $M_4[\Gamma_0(4)]$ and $M_2[\Gamma_0(2)]$ are finite dimensional. Notably, however, since the generating functions for $\sigma_a$ for $a$ even are odd weight Eisenstein series, the analysis of the  convolution of $S_{a,b}$ for even $a$ resists the theory of quasimodular forms, and hence we use alternate methods to understand the asymptotics of such sums. 

We now describe the specific problem in the enumeration of STSs that motivated us to study $S_{a,b}$ for even $a$.  

Every STS can be viewed as a union of horizontal square-tiled cylinders glued together. One way to analyze an STS in a given stratum is to categorize its horizontal cylinder decomposition type, popularly termed \emph{cylinder diagram} that describes how many horizontal cylinders makes up the surface, and in what ways they are glued together. 

In particular, STSs in $\calH(1,1)$ (translation surfaces of genus two with two cone points) partition into exactly 4 cylinder diagrams. Figure \ref{fig:4types} shows prototypical examples of surfaces in the 4 cylinder diagrams named A, B, C and D in $\calH(1,1)$.
\begin{figure}[h!!]
\begin{tabular}{ccccccc}
\begin{tikzpicture}[scale=0.32, sq/.style=
  {shape=regular polygon, regular polygon sides=4, draw, minimum width=0.1mm}]

\draw (0,0) -- (11,0) -- (11,3) -- (0,3) -- (0,0);

\foreach \i in {0,5,11}{
\draw[fill] (\i,0) circle [radius = 0.12];
}
\foreach \i in {0,6,11}{
\draw[fill] (\i,3) circle [radius = 0.12];
}

\foreach \i in {2,7}{
\draw[fill] (\i,0) node [scale = 0.5, sq]{};
}
\foreach \i in {2,8}{
\draw (\i,3) node [scale = 0.5, sq]{};
}

\foreach \i in {1,2}{
\draw (0,\i) -- (11, \i);
}

\foreach \i in {1,2,3,4,5,6,7,8,9,10}{
\draw (\i,0) -- (\i, 3);
}

\draw (-0.5,1.5) node {$p$};
\draw (11.5,1.5) node {$p$};

\begin{scope}[shift = {(0,-0.2)}]

\draw  (1,-0.35) node {$j$};
\draw  (3.5,-0.35) node {$k$};
\draw  (6,-0.35) node {$\ell$};
\draw  (9,-0.35) node {$m$}; 
\end{scope}

\begin{scope}[shift={(0,0.2)}]

\draw (1,3.35) node {$j$};
\draw (4,3.35) node {$m$};

\draw  (7,3.35) node {$\ell$};

\draw (9.5,3.35) node {$k$};

\end{scope}

\end{tikzpicture}&\hspace{0.01cm} &

\begin{tikzpicture}[scale=0.32,sq/.style=
  {shape=regular polygon, regular polygon sides=4, draw, minimum width=0.1mm}]
\draw (0,0) -- (11,0) --(11, 2) -- (3,2) -- (3,5) -- (0,5) -- (0,0);

\foreach \i in {1}{
\draw (0,\i) -- (11, \i);
}

\foreach \i in {1,2,3,4,5,6,7,8,9,10}{
\draw (\i,0) -- (\i, 2);
}

\begin{scope}[shift ={(0,2)}]
\foreach \i in {0,1,2}{
\draw (0,\i) -- (3, \i);
}

\foreach \i in {1,2}{
\draw (\i,0) -- (\i, 3);
}

\end{scope}

\foreach \i in {0,3,11}{
\draw[fill] (\i,0) circle [radius = 0.12];
}
\foreach \i in {8}{
\draw[fill] (\i,2) circle [radius = 0.12];
}
\foreach \i in {0,3}{
\draw[fill] (\i,5) circle [radius = 0.12];
}

\foreach \i in {0,3, 11}{
\draw[fill] (\i,2) node [scale = 0.5, sq]{};
}
\foreach \i in {6}{
\draw (\i,0) node [scale = 0.5, sq]{};
}

\draw (-0.5,1) node {$p$};
\draw (11.5,1) node {$p$};

\draw (-0.5,3.5) node {$q$};
\draw (3.5,3.5) node {$q$};

\begin{scope}[shift={(0,-0.2)}]
\draw (1.5,-0.35) node {$m$};
\draw (4.5,-0.35) node {$\ell$};
\draw (8.5,-0.35) node {$k$};
\end{scope}

\begin{scope}[shift = {(0,0.2)}]
\draw (5.5, 2.35) node {$k$};
\draw (9.5, 2.35) node {$\ell$};

\draw (1.5, 5.35) node {$m$};;

\end{scope}

\end{tikzpicture}& &

\begin{tikzpicture}[scale=0.32, sq/.style=
  {shape=regular polygon, regular polygon sides=4, draw, minimum width=0.1mm}]
\draw (2,2) -- (5,2) --(5, 0) -- (10,0) -- (10,2) -- (8,2) -- (8,5) -- (2,5) -- (2,2);

\foreach \i in {1}{
\draw (5,\i) -- (10, \i);
}

\foreach \i in {6,7,8,9}{
\draw (\i,0) -- (\i, 2);
}

\begin{scope}[shift ={(2,2)}]
\foreach \i in {1,2}{
\draw (0,\i) -- (6, \i);
}

\foreach \i in {1,2,3,4,5}{
\draw (\i,0) -- (\i, 3);
}

\draw (3,0) -- (6,0);

\end{scope}

\foreach \i in {5,10}{
\draw[fill] (\i,0) circle [radius = 0.12];
}
\foreach \i in {5, 10}{
\draw[fill] (\i,2) circle [radius = 0.12];
}
\foreach \i in {5}{
\draw[fill] (\i,5) circle [radius = 0.12];
}

\foreach \i in {2,8}{
\draw (\i,2) node [scale = 0.5, sq]{};
}
\foreach \i in {2,8}{
\draw (\i,5) node [scale = 0.5, sq]{};
}

\draw (8,0) node [scale = 0.5, sq]{};

\draw (10.5,1) node {$p$};
\draw (4.5,1) node {$p$};

\draw (8.5,3.5) node {$q$};
\draw (1.5,3.5) node {$q$};

\begin{scope}[shift = {(0,-0.2)}]

\draw (3.5,2-0.35) node{$m$};
\draw (6.5, -0.35) node{$\ell$};
\draw (9, -0.35) node{$k$};
\end{scope}

\begin{scope}[shift = {(0, 0.2)}]

\draw (3.5,5.35) node{$m$};
\draw (6.5, 5.35)  node{$\ell$};
\draw (9, 2.35) node {$k$};

\end{scope}

\end{tikzpicture} & & \begin{tikzpicture}[scale=0.32, sq/.style=
  {shape=regular polygon, regular polygon sides=4, draw, minimum width=0.1mm}]
\draw (1,2) -- (4,2) -- (4,0) --(6, 0) -- (6,2) -- (6,4) -- (4,4) -- (4,7) -- (1,7) -- (1,2);

\foreach \i in {1,2}{
\draw (4,\i) -- (6, \i);
}
\foreach \i in {5,6}{
\draw (\i,0) -- (\i, 2);
}

\foreach \i in {1, 4,6}{
\draw[fill] (\i,2) circle [radius = 0.12];
}
\foreach \i in {1,4}{
\draw[fill] (\i,7) circle [radius = 0.12];
}

\foreach \i in {4,6}{
\draw[fill] (\i,0) node [scale = 0.5, sq]{};
}
\foreach \i in {1,4,6}{
\draw (\i,4) node [scale = 0.5, sq]{};
}

%% mid cylinder
\begin{scope}[shift = {(1,2)}]
\foreach \i in {1}{
\draw (0,\i) -- (5, \i);
}

\foreach \i in {1,2,3,4}{
\draw (\i,0) -- (\i, 2);
}

\end{scope}

%% top cylinder
\begin{scope}[shift ={(1,4)}]
\foreach \i in {0,1,2,3}{
\draw (0,\i) -- (3, \i);
}

\foreach \i in {1,2}{
\draw (\i,0) -- (\i, 3);
}

\end{scope}

\draw (6.5,1) node {$p$};
\draw (3.5,1) node {$p$};

\draw (6.5,3) node {$q$};
\draw (0.5,3) node {$q$};

\draw (4.5,5.5) node {$r$};
\draw (0.5,5.5) node {$r$};

\begin{scope}[shift = {(0,-0.2)}]
\draw (5,-0.35) node {$k$};
\draw (2.5,2-0.35) node {$\ell$};
\end{scope}

\begin{scope}[shift = {(0,0.2)}]
\draw (5,4+0.35) node {$k$};
\draw (2.5,7+0.35) node {$\ell$};

\end{scope}

\end{tikzpicture}\\
A & \hspace{0.1cm} & B & \hspace{0.1cm} & C & \hspace{0.1cm}  & D

\end{tabular}
\caption{Examples of STSs in the four cylinder diagrams of $\calH(1,1)$, here named A, B, C, D. In each surface, collections of edges with the same label are glued via translation. For instance, in A, the 3 edges labeled $p$ are glued to the 3 edges labeled $p$ via translation to form a horizontal cylinder. Hence, diagram A is characterized by having exactly one (maximal) horizontal cylinder. Similarly, diagram D consists of STSs in $\calH(1,1)$ with exactly three horizontal cylinders while diagram B and C consist of those with two horizontal cylinders but different gluing pattern. Adding squares to vary the parameters $p,q, r, j, k, l, m$ gives surfaces with different number of squares in each of these cylinder diagrams. }

\label{fig:4types}
\end{figure}
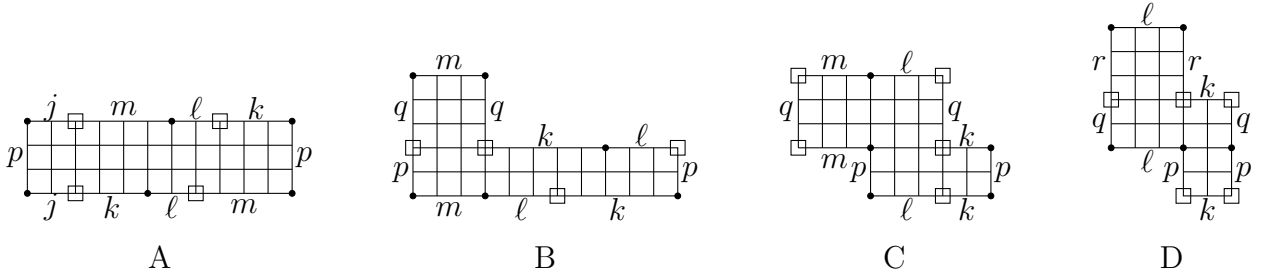

The counting problem in question is to enumerate, given a fixed $n$, the number of primitive STSs in $\calH(1,1)$ in each of the four cylinder diagrams and find the individual asymptotic densities of each them. For example, let the number of primitive $n$-square surfaces in $\calH(1,1)$ with diagram D be $D(n)$. The second author proved in \cite{Shres} that 
$$D(n) = \frac{1}{6}n(n-1)J_2(n) - \bigl((\mu \cdot \sigma_2) * (S_{1,2})\bigr)(n),$$
where $J_k(n) := n^k \prod_{p | n} \left(1 - \frac{1}{p^k}\right)$ is the Jordan totient function of order $k$, $\mu$ is the M\"{o}bius function and $*$ is Dirichlet convolution. Using Theorem \ref{thm:main}, 
the second author proved that surfaces with diagram D have asymptotic density $1- \frac{\zeta(2)\zeta(3)}{2\zeta(5)} \approx 0.047$. For similar formulae and asymptotic densities concerning the other diagrams A, B and C, see \cite[Theorem 1.1]{Shres}.

An analogous problem for the other genus two stratum $\calH(2)$ was solved by Zmiaikou \cite{Zmi}. Complete results for strata of genus 3 and above are not known although the density of one cylinder surfaces (although not necessarily primitive) has been computed by Delecroix-Goujard-Zograf-Zorich \cite{DelGouZogZor}.

\section{Proof of Theorem \ref{thm:main}}\label{sec:elem}

For the reader's convenience, we begin with a short proof of Ramanujan's conjecture, along similar lines to 
Halberstam \cite{halberstam}:

\begin{theorem}\label{thm:main}
For any positive real numbers $a$ and $b$, as $n\to\infty$ there holds
\begin{equation}\label{eq:main}
S_{a,b}(n) \sim \frac{\Gamma(a+1)\Gamma(b+1)}{\Gamma(a+b+2)} \frac{\zeta(a+1)\zeta(b+1)}{\zeta(a+b+2)} \sigma_{a+b+1}(n).
\end{equation}
\end{theorem}
As with \cite{halberstam} we will obtain a power saving error term. The theorem also holds if $a$ and $b$ are complex numbers with positive real part, in which case replace $a$ and $b$ by their real parts
everywhere in the error terms and inequalities.

We begin with two lemmas.

\begin{lemma}\label{lem:riemann_sum}
For any integer $n$ and residue class $k \pmod{m}$, we have
\begin{equation}\label{eq:riemann_result}
\sum_{\substack{j = 1 \\ j \equiv k \pmod m}}^{n - 1} j^a (n - j)^b = \frac{n^{a + b + 1}}{m} \frac{\Gamma(a+1)\Gamma(b+1)}{\Gamma(a+b+2)}
+ O_{a, b}\left(n^{a + b} \right).
\end{equation}
\end{lemma}
\begin{proof} (Sketch) As in \cite{halberstam}, we rewrite the sum in \eqref{eq:riemann_result} as
$n^{a + b} \sum_{j = 0}^{r-1} f\left(\alpha_0 + j\alpha\right)$,
where $f(t) := t^a (1 - t)^b$,
for some $\alpha_0$ and $r$ satisfying $0 \leq \alpha_0 < \alpha$ and $\left| r - \alpha^{-1} \right| < 1$.
After a change of variables, we recognize this as a Riemann sum approximation to the integral defining the beta function,
yielding the result.
\end{proof}

\begin{lemma}\label{lem:dir_identity}
We have, as a formal identity of Dirichlet series,
\[
\sum_{n = 1}^{\infty} \sum_{\substack{m = 1 \\ (m, n) = 1}}^{\infty}
n^{-r} m^{-s} =
\frac{\zeta(r) \zeta(s)}{\zeta(r + s)}.
\]
\end{lemma}
\begin{proof}
This follows by rewriting the left side as
\begin{align*}
\sum_{d = 1}^{\infty} \mu(d) \sum_{u = 1}^{\infty} \sum_{v = 1}^{\infty}
(du)^{-r} (dv)^{-s}
= 
\sum_{d = 1}^{\infty} \mu(d) d^{-r- s} \sum_{u = 1}^{\infty} \sum_{v = 1}^{\infty}
u^{-r} v^{-s}.
\end{align*}
\end{proof}

\begin{proof}[Proof of Theorem \ref{thm:main}]
We rewrite $S_{a, b}(n)$ in the form
\begin{equation}\label{eq:concl2}
S_{a, b}(n) = \sum_{k=1}^{n-1} \sigma_a(k) \sigma_b(n-k) = 
\sum_{d=1}^{n - 1} d^{-a}
\sum_{e=1}^{n - 1} e^{-b}
\sum_{\substack{k = 1 \\ d \mid k \\ e \mid n - k}}^{n - 1} k^a (n - k)^b.
\end{equation}
If $(d, e) \nmid n$ then the inner sum vanishes. Otherwise, the divisibility
conditions are equivalent to demanding that 
$k \equiv k_0 \pmod{\frac{de}{(d, e)}}$ for some $k_0$, and by
Lemma \ref{lem:riemann_sum} the inner sum equals
\begin{align*}
&
n^{a + b + 1} \left( \frac{de}{(d, e)} \right)^{-1}  \frac{\Gamma(a+1)\Gamma(b+1)}{\Gamma(a+b+2)} + 
O\left(n^{a + b} \right),
\end{align*}
so that 

\begin{equation}\label{eq:concl3}
S_{a, b}(n) = \frac{\Gamma(a+1)\Gamma(b+1)}{\Gamma(a+b+2)}  n^{a + b + 1}
\sum_{\substack{d, e=1 \\ (d, e) \mid n}}^{n - 1} d^{-a} e^{-b}
\left( \frac{(d, e)}{de} + O(n^{-1}) \right).
\end{equation}

Assuming for now that $a, b > 1$, the error term of $O(n^{-1})$ above contributes an error bounded by
\begin{equation}\label{eq:de_bound}
\ll n^{a + b} \sum_{d, e = 1}^{n - 1} d^{-a} e^{-b} 
\ll n^{a + b}.
\end{equation}

The sum in the main term of \eqref{eq:concl3} is equal to
\begin{align*}
\sum_{w \mid n} &w^{-a - b - 1} 
\sum_{\substack{i, j = 1 \\ (i, j) = 1}}^{n/w-1}  i^{-a - 1} j^{- b - 1} \\
	& =  \sum_{w \mid n} w^{-a - b - 1} 
		\left(\sum_{\substack{i, j = 1 \\ (i, j) = 1}}^{\infty}  i^{-a - 1} j^{- b - 1} 
		+ O\left( \left(\frac{n}{w} \right)^{- \min(a, b)} \right) \right) \\
	& = \sum_{w \mid n} w^{-a - b - 1}  \sum_{\substack{i, j = 1 \\ (i, j) = 1}}^{\infty}  i^{-a - 1} j^{- b - 1} + O(n^{- \min(a, b)}). \\
	\end{align*}
By Lemma \ref{lem:dir_identity} the sum over $i$ and $j$ above is $\zeta(a+1)\zeta(b+1)/\zeta(a+b+2)$, while the sum over $w$ may be identified as $n^{-a-b-1}\sigma_{a+b+1}(n)$.  Assembling this in \eqref{eq:concl3}, we obtain Theorem 1.1 with an error of $O(n^{a + b})$ in the case that $a,b>1$.

\medskip
If $a \leq 1$ and $b \geq 1$, then in \eqref{eq:de_bound} the error term
is $\ll n^{a + b + 1 - a} (\log n)^2$, the logarithmic factors being relevant only if $a=1$ or $b=1$.

If instead  $a, b < 1$, take the sum in \eqref{eq:de_bound} only through
$d \leq D$ and $e \leq E$, making an error $\ll n^{a + b} D^{1 - a} E^{1 - b}$.
Rewriting \eqref{eq:concl2} in the form
\begin{equation}\label{eq:concl2a}
\sum_{k = 1}^{n - 1} O(n^{a + b})
\left( \sum_{d \mid k} d^{-a} \right)
\left( \sum_{e \mid n - k} e^{-b} \right),
\end{equation}
the contribution from $d > D$ is $O(n^{a + b + 1 + \epsilon} D^{-a})$, and the contribution
from $e > E$ is similarly $O(n^{a + b + 1 + \epsilon} E^{-b})$. We therefore make a total error
\[
\ll 
n^{a + b + 1 + \epsilon} \max( n^{-1} D^{1 - a} E^{1 - b}, D^{-a}, E^{-b} ).
\]
Equating the parameters by choosing $D = n^{\frac{b}{b+a-ab}}$ and $E = n^{\frac{a}{b+a-ab}}$,
we obtain an error term
\[
\ll
n^{a + b + 1 + \epsilon - \frac{ab}{b+a-ab}}.
\]
This yields Theorem \ref{thm:main} in the remaining cases.
\end{proof}

%%%%%%%%%%%%%%%%%%%%%%%
%%%%%%%%%%%%%%%%%%%%%%%
%%%%%%%%%%%%%%%%%%%%%%%
\section{Main theorem and proof}\label{sec:st}
%%%%%%%%%%%%%%%%%%%%%%%
%%%%%%%%%%%%%%%%%%%%%%%
%%%%%%%%%%%%%%%%%%%%%%%

Again, for notational simplicity we assume that $b$ and $a$ are both real; if not, replace $b$ and $a$ with $\Re(b)$ and $\Re(a)$
in all inequalities and error estimates. 
We also assume without loss of generality that $b \geq a$ (i.e., that $\Re(b) \geq \Re(a)$ if these quantities are complex).

To motivate our strategy, in place of $\sum_{k = 1}^{n - 1} \sigma_a(k) \sigma_b(n - k)$, consider the problem
of estimating the simpler sum $\sum_{k = 1}^{n - 1} \sigma_a(k) (n - k)^b$. The  factor
$(n - k)^b$ appears to complicate matters, but
via the theory of Riesz means and Mellin transforms it may be interpreted as a {\itshape smoothing} factor that {\itshape helps} in evaluating of the sum.

In particular, we have the following familiar formula.

\begin{lemma}
	We have, for any Dirichlet series $\sum_k a(k) k^{-s}$ and any complex number $b$ with $\Re(b) > 0$, the formula
	\begin{equation}\label{eq:riesz}
		\frac{1}{\Gamma(b + 1)} \sum_{k = 1}^n a(k) (n - k)^b = \frac{1}{2 \pi i} \int \left( \sum a(k) k^{-s} \right) \frac{\Gamma(s)}{\Gamma(s + b + 1)} n^{s + b} ds,
	\end{equation}
	where the contour is over any vertical line where the Dirichlet series converges uniformly and absolutely.
\end{lemma}

\begin{proof} 
	Switching the order of integration and summation, this reduces to the formula
	\[
	 \frac{1}{2 \pi i} \int \frac{\Gamma(s)}{\Gamma(s + b + 1)} t^{s} ds =
	 \begin{cases} 0 & \textnormal{ if } 0 < t < 1, \\ \Gamma(b + 1)^{-1} \cdot (1 - t^{-1})^b  & \textnormal{ if }   t > 1,
	 \end{cases}
	 \]
	 for which see \cite[17.43.22]{GR}. (It may be proved by shifting the contour infinitely far to the right or left as appropriate, and evaluating the sum
	 of residues in the latter case.)
\end{proof}

Our goal will be to first manipulate our sum into something resembling \eqref{eq:riesz},
where the Dirichlet series $\sum a(k) k^{-s}$ can be expressed in terms of zeta functions and therefore enjoys analytic continuation to $\mathbb{C}$. As is familiar
in various analytic number theory contexts, this will then allow us to shift the integral in \eqref{eq:riesz} to the left. 

Now, we have
\begin{align}
	S_{a, b}(n) = \nonumber  
	\sum_{k = 1}^{n - 1} \sigma_a(k) \sigma_b(n - k) & = 
	\sum_{k = 1}^{n - 1} \sigma_a(k) \left( \sum_{d \mid n - k} \left( \frac{n -k}{d} \right)^b \right) \\ \nonumber
	& = \sum_{d \geq 1} d^{-b} \sum_{\substack{k = 1 \\ d \mid n - k}}^{n - 1} \sigma_a(k) (n - k)^b \\
	& = \Gamma(b + 1)  \sum_{d \geq 1} d^{-b} \frac{1}{2 \pi i} \int_{(a + 2)} \left( \sum_{\substack{k \\ d \mid n - k}} \sigma_a(k) k^{-s} \right) 
	\frac{\Gamma(s)}{\Gamma(s + b + 1)} 
	n^{b + s} ds, \label{eq:sab_second}
\end{align}
\medskip
where the integral is taken over the vertical line with $\Re(s) = a + 2$.

For any real $x>0$, let $\zeta(s,x)$ be the Hurwitz zeta function, defined for $\Re(s)>1$ by the Dirichlet series
\[
	\zeta(s,x)
		:= \sum_{n=0}^\infty \frac{1}{(n+x)^s}.
\]
We note that \begin{align*}
	\sum_{k \equiv n \pmod{d}} \frac{\sigma_a(k)}{k^s}
		&= \sum_{k_1k_2 \equiv n \pmod{d}} \frac{k_1^a}{(k_1k_2)^s} \\
		&= \sum_{\substack{ 1 \leq e_1,e_2 \leq d \\ e_1 e_2 \equiv n \pmod{d}}} \left(\sum_{m_1 \geq 0} \frac{1}{(m_1d + e_1)^{s-a}} \right) \left( \sum_{m_2 \geq 0} \frac{1}{(m_2d+e_2)^s}\right) \\
		&= \frac{1}{d^{2s-a}} \sum_{\substack{ 1 \leq e_1,e_2 \leq d \\ e_1 e_2 \equiv n \pmod{d}}} \zeta(s-a, e_1/d) \zeta(s,e_2/d).
\end{align*}
Thus, we conclude that
\begin{equation} \label{eqn:sab-hurwitz}
	S_{a,b}(n)
		= \Gamma(b+1) \sum_{d \geq 1} d^{a-b} \sum_{\substack{1 \leq e_1,e_2 \leq d \\ e_1e_2 \equiv n \pmod{d}}} \frac{1}{2\pi i} \int_{(a+2)} \zeta(s-a,e_1/d) \zeta(s,e_2/d) \frac{\Gamma(s)}{\Gamma(s+b+1)}n^{b+s}d^{-2s}\,ds.
\end{equation}

The main goal of this section is to prove the following theorem, essentially a restatement of Theorems \ref{thm:st1} and \ref{thm:st2}.

\begin{theorem}\label{thm:st-hurwitz}
	Let $a$ and $b$ be positive real numbers.  

	1.  If $b-a > 3/2$, then
	\begin{align*}
		S_{a,b}(n) 
			&= \frac{\Gamma(a+1)\Gamma(b+1)}{\Gamma(a+b+2)}\frac{\zeta(a+1)\zeta(b+1)}{\zeta(a+b+2)}\sigma_{a+b+1}(n) \\ 
			\notag &\quad  + \frac{\zeta(1-a)\zeta(b+1)}{(b+1)\zeta(b-a+2)}n^a\sigma_{b-a+1}(n) 
			 + \sum_{0 \leq m < \frac{b-a}{2}-\frac{3}{4}} \mathrm{Res}(-m) + O_\epsilon(n^{\frac{a+b}{2}+\frac{3}{4}+\epsilon}),
	\end{align*}
	where $\mathrm{Res}(-m)$ denotes the residue of the integrand of \eqref{eqn:sab-hurwitz} at $s=-m$, and is given explicitly by \eqref{eqn:s=-m}.  It satisfies $\mathrm{Res}(-m) \ll n^{b-m}$ in general, and if $a$ is an odd integer, then $\mathrm{Res}(0) = -\frac{1}{2}\zeta(-a)\sigma_b(n)$ and $\mathrm{Res}(-m)=0$ for each $m \geq 1$.

	2.  If $\max\{a,2-a\} < b \leq a + \frac{3}{2}$, then
	\begin{align*}
		S_{a,b}(n) 
			&= \frac{\Gamma(a+1)\Gamma(b+1)}{\Gamma(a+b+2)}\frac{\zeta(a+1)\zeta(b+1)}{\zeta(a+b+2)}\sigma_{a+b+1}(n)  \\
			& \quad\quad + \frac{\zeta(1-a)\zeta(b+1)}{(b+1)\zeta(b-a+2)}n^a\sigma_{b-a+1}(n) + O(n^{\frac{a+b}{2}+1+\epsilon}).
	\end{align*}
\end{theorem}

After recalling some analytic facts about the Hurwitz zeta function,
we begin by analyzing the poles and residues of the integrand.  This constitutes an analysis of the main terms provided in Theorem \ref{thm:st-hurwitz}.  We then bound the error terms in Theorem \ref{thm:st-hurwitz} by means of the functional equation for Hurwitz zeta functions.  This has the net effect of replacing the summation of Hurwitz zeta functions by a Dirichlet series whose coefficients are certain Kloosterman sums.  This also implicitly gives another evaluation of the residual terms $\mathrm{Res}(-m)$.

Finally, we note that we can obtain the secondary term in a simpler fashion, with no Kloosterman sums,
when $a > 1$ and $b > a + 2$. We explain this in Section \ref{sec:easy}.

\subsection{Properties of the Hurwitz zeta function}
The following lemma recalls some basic properties of the Hurwitz zeta function. For proofs, see \cite{Apostol}.
\begin{lemma}
	For any real $x > 0$, the Hurwitz zeta function $\zeta(s, x) := \sum_{n = 0}^{\infty} (n + x)^{-s}$ satisfies the following:
	\begin{itemize}
		\item {\upshape (Analytic continuation)} 
			$\zeta(s, x)$ has analytic continuation to all of $\mathbb{C}$, with a simple pole at $s = 1$ with residue $1$, and holomorphic elsewhere.
		\item {\upshape (Functional equation)} 
			$\zeta(s, x)$ satisfies a functional equation, which for $x = e/d$ rational can be written
			\begin{equation} \label{eqn:hurwitz-fe}
				\zeta(1-s,e/d)
					= \frac{\Gamma(s)}{(2\pi)^s} \left( e^{\pi i s/2} \sum_{k \geq 1} \frac{e^{-2\pi i ke/d}}{k^s} + e^{-\pi i s/2} \sum_{k \geq 1} \frac{e^{2\pi i ke/d}}{k^s}\right).
			\end{equation}
		\item {\upshape (Evaluation at negative integers)}
			For integer values $k \geq 0$, there is the special value
			\begin{equation}\label{eqn:zeta_neg_int}
				\zeta(-k, x) = \frac{-1}{k+1} B_{k+1}(x),
			\end{equation}
			where $B_{k+1}(x)$ denotes the degree $k+1$ Bernoulli polynomial. 
	\end{itemize}
\end{lemma}

To estimate the values of $\zeta(s, x)$ inside the critical strip, we will use the {\itshape approximate functional equation}, as proved in the following form by Miyagawa \cite{miyagawa}.

\begin{lemma}\label{lem:AFE}
	Assume $s = \sigma + i t$ for some $0 < \sigma < 1 $.  Set $T = \sqrt{2 \pi (|t|+1)}$.  Then for any real $x > 0$,
	\begin{align*}
		\zeta&(s,x) = \\
			&\sum_{0 \leq k \leq T} \frac{1}{(k+x)^s} 
			+ \frac{\Gamma(1-s)}{(2\pi)^{1-s}} \left[ e^{\frac{\pi i (1-s)}{2}}\sum_{k \leq T} \frac{e(-kx)}{k^{1-s}} + e^{\frac{-\pi i (1-s)}{2}}\sum_{k \leq T} \frac{e(kx)}{k^{1-s}}\right]
			+ O(t^{-\frac{\sigma}{2}}) + O(t^{\frac{\sigma-1}{2}}).
	\end{align*}
\end{lemma}

We also note the following consequence of Stirling's formula.
\begin{lemma}\label{lem:gamma_quotient} 
	For any $b$, we have
	\[
		\frac{\Gamma(s)}{\Gamma(1 + b + s)} \ll_b (1 + |t|)^{-b - 1}.
	\]
\end{lemma}

%%%%%%%%%%%%%%%%%%%%%%%%%%
\subsection{Analysis of poles and residues}
%%%%%%%%%%%%%%%%%%%%%%%%%%

We now proceed with our analysis of the integral \eqref{eqn:sab-hurwitz}.
For each $e_1,e_2$, the integrand has right-most pole at $s=a+1$, coming from the factor of $\zeta(s-a,e_1/d)$, which has a simple pole with residue $1$.  The sum of the residues is
\begin{align*}
\frac{\Gamma(a+1)\Gamma(b+1)}{\Gamma(a+b+2)} \sum_{d \geq 1} &\frac{n^{a+b+1}}{d^{a+b+2}}  \sum_{\substack{1 \leq e_1,e_2 \leq d \\ e_1e_2 \equiv n \pmod{d}}} \zeta(a+1,e_2/d) \\
	&= \frac{\Gamma(a+1)\Gamma(b+1)}{\Gamma(a+b+2)} \sum_{d \geq 1} \frac{n^{a+b+1}}{d^{b+1}} \sum_{k \geq 1} \frac{\#\{e_1 \pmod{d} : ke_1 \equiv n \pmod{d}\}}{k^{a+1}}.
\end{align*}
We then note that 
\[
	\#\{e_1 \pmod{d} : ke_1 \equiv n \pmod{d}\}
	= \begin{cases}
		(k,d), & \text{if } (k,d) \mid (d,n) \\
		0, & \text{otherwise.}
	\end{cases}
\]
Thus, write $f := (k,d)$, and observe that we may assume $f \mid n$.  So doing, and replacing $d$ and $k$ by $fd$ and $fk$, respectively, our expression for the residue at $s=a+1$ becomes
\begin{align*}
	\frac{\Gamma(a+1)\Gamma(b+1)}{\Gamma(a+b+2)} \sum_{f \mid n} \frac{n^{a+b+1}}{f^{a+b+1}}
		\sum_{\substack{d,k \\ (d,k)=1}} \frac{1}{d^{b+1} k^{a+1}}
		&= \frac{\Gamma(a+1)\Gamma(b+1)}{\Gamma(a+b+2)} \frac{\zeta(a+1)\zeta(b+1)}{\zeta(a+b+2)} \sigma_{a+b+1}(n),
\end{align*}
by Lemma \ref{lem:dir_identity}.

Before turning to the residue of the pole at $s=1$, we note one consequence of the above argument.  In particular, for any fixed $n$ and $b$, in the identity proved above,
\begin{equation}\label{eqn:hurwitz-identity}
	\sum_{d \geq 1} \frac{n^{a+b+1}}{d^{a+b+2}} \sum_{\substack{1 \leq e_1,e_2 \leq d \\ e_1e_2 \equiv n \pmod{d}}} \zeta(a+1,e_2/d)
		= \frac{\zeta(a+1)\zeta(b+1)}{\zeta(a+b+2)} \sigma_{a+b+1}(n), 
\end{equation}
both sides define analytic functions of $a$ for $\Re(a)>-b$, $a\neq 0$.  Thus, this expression must hold for $-b<\Re(a)<0$, even though neither $\zeta(a+1,x)$ nor $\zeta(a+1)$ is defined via a convergent Dirichet series in this region.  This will be useful in evaluating the residue at $s=1$, which we now turn to.

Using \eqref{eqn:sab-hurwitz} again, the pole at $s=1$ is seen to be
\begin{align*}
	\frac{\Gamma(b+1)}{\Gamma(b+2)} \sum_{d\geq 1} \frac{n^{b+1}}{d^{b-a+2}} \sum_{\substack{1 \leq e_1,e_2 \leq d \\ e_1e_2 \equiv n \pmod{d}}} \zeta(1-a,e_1/d)
		&= \frac{n^a}{b+1} \sum_{d \geq 1} \frac{n^{b-a+1}}{d^{b-a+2}} \sum_{\substack{1 \leq e_1,e_2 \leq d \\ e_1e_2 \equiv n \pmod{d}}} \zeta(1-a,e_1/d).
\end{align*}
Since we have assumed $a < b$, it follows that $-a > -b$, so by the identity \eqref{eqn:hurwitz-identity}, this evaluates to
\[
	\frac{n^a}{b+1} \frac{\zeta(1-a)\zeta(b+1)}{\zeta(b-a+2)}\sigma_{b-a+1}(n).
\]

Finally, we evaluate the residue at $s=-m$, $m \geq 0$, arising from the gamma function.  We do so in general, but we only provide a clean simplification of the term when $a$ is an odd integer.  The residues for other values of $a$ do not seem to have a natural multiplicative structure, for example, so we consider the case that $a$ is odd to be the most interesting.

Using \eqref{eqn:sab-hurwitz}, the residue at $s=-m$ is
\begin{equation}\label{eqn:s=-m}
	(-1)^m n^{b-m} \left({b}\atop {m}\right) \sum_{d \geq 1} d^{a-b+2m} \sum_{\substack{1 \leq e_1,e_2 \leq d \\ e_1e_2 \equiv n \pmod{d}}} \zeta(-m-a,e_1/d) \zeta(-m,e_2/d).
\end{equation}

When $a$ is an integer,
by the special value formula \eqref{eqn:zeta_neg_int} the inner summation over $e_1,e_2$ in \eqref{eqn:s=-m} becomes
\[
	\frac{1}{(m+1)(m+a+1)} \sum_{\substack{ 1 \leq e_1,e_2 \leq d \\ e_1e_2 \equiv n \pmod{d}}} B_{m+1}(e_2/d) B_{m+a+1}(e_1/d).
\]
For fixed $d$, the substitution $(e_1,e_2) \mapsto (d-e_1,d-e_2)$ defines an involution on the set of pairs $(e_1,e_2)$ with $e_1,e_2 \neq d$.  Since $B_{k+1}(1-x) = (-1)^{k+1}B_{k+1}(x)$, if $a$ is odd, it follows for such $e_1,e_2$ that 
\[
	B_{m+1}\Big(\frac{d-e_2}{d}\Big)B_{m+a+1}\Big(\frac{d-e_1}{d}\Big)
		= - B_{m+1}\Big(\frac{e_2}{d}\Big) B_{m+a+1}\Big(\frac{e_1}{d}\Big).
\]
Consequently, when $a$ is odd, the sum over $e_1,e_2$ with $e_1,e_2 \neq d$ cancels, and it remains to consider only those pairs where one of $e_1$ and $e_2$ equals $d$.  Given that $e_1$ and $e_2$ are restricted to satisfy the congruence $e_1e_2 \equiv n \pmod{d}$, such pairs arise only when $d \mid n$.  In this case, the summation over $e_1$ and $e_2$ in \eqref{eqn:s=-m} collapses to
\begin{align*}
	\sum_{e_1 = 1}^d \zeta(-m-a,e_1/d)&\zeta(-m) + \sum_{e_2 =1}^d \zeta(-m,e_2/d) \zeta(-m-a) - \zeta(-m-a)\zeta(-m) \\
		&= d^{-m-a} \zeta(-m-a)\zeta(-m) + d^{-m} \zeta(-m)\zeta(-m-a) - \zeta(-m)\zeta(-m-a).
\end{align*}
If $m \geq 1$, then, since $a$ is odd, every term above is $0$, and consequently the residue \eqref{eqn:s=-m} is $0$ as well.  On the other hand, if $m=0$, then the above expression simplifies to $d^{-a}\zeta(0)\zeta(-a) = -\frac{d^{-a}}{2}\zeta(-a)$.  We then find for $m=0$ that \eqref{eqn:s=-m} evaluates to
\[
	\frac{-\zeta(-a)}{2} \sum_{d \mid n} \frac{n^b}{d^b}
		= -\frac{\zeta(-a)}{2} \sigma_b(n).
\]

\subsection{Error analysis via Kloosterman sums}

Applying the functional equation \eqref{eqn:hurwitz-fe} for both $\zeta(1-s-a,e_1/d)$ and $\zeta(1-s,e_2/d)$, we will be led to consider exponential sums of the form
\[
	S_n(m,k;d)
		:= \sum_{\substack{e_1,e_2 \pmod{d} \\ e_1 e_2 \equiv n \pmod{d}}} e\Big( \frac{me_1+ke_2}{d}\Big),
\]
where we write $e(x) := e^{2\pi i x}$ for any real $x$.  By relating these to classical Kloosterman sums 
we obtain the following strong bound.

\begin{lemma}\label{lem:kloosterman}
With notation as above, we have
\[
S_n(m,k;d) \ll_\epsilon d^{1/2+\epsilon} (d,k)^{1/2} (d,m)^{1/2}\]
for any $\epsilon > 0$.
\end{lemma}
\begin{proof}
Recall that the classical Kloosterman sums are defined by
\[
	K(a,b;q) := S_1(a,b;q) = \sum_{ xy \equiv 1 \pmod{q}} e\left(\frac{ax+by}{q}\right).
\]
We begin by proving the identity
\[
	S_n(m,k;d)
		= \sum_{f \mid (d,n,k)} f ~K(m, kn/f^2; d/f),
\]

For $e_1$ as in the sum defining $S_n(m,k;d)$, let $f = (e_1,d)$, and note that there are no terms with $f \nmid (d,n)$.  Write $e_1 = e_1^\prime f$, where $(e_1^\prime, d/f) = 1$.  Let $e_2^\prime$ be such that $e_1^\prime e_2^\prime \equiv 1\pmod{d/f}$, so that the allowed values of $e_2 \pmod{d}$ are given by $e_2 = e_2^\prime n/f + jd/f$ for $0 \leq j \leq f-1$.

Thus, we find
\begin{align*}
S_n(m,k;d)
	&= \sum_{f \mid (d,n)} \sum_{e_1^\prime e_2^\prime \equiv 1 \pmod{\frac{d}{f}}} e\left(\frac{m e_1^\prime f + kne_2^\prime/f}{d}\right)\sum_{j=0}^{f-1} e\left(\frac{jk}{f}\right) \\
	&= \sum_{f \mid (d,n,k)} f \sum_{e_1^\prime e_2^\prime \equiv 1 \pmod{\frac{d}{f}}} e\left(\frac{me_1^\prime + kne_2^\prime /f^2}{d/f}\right) \\
	&= \sum_{f \mid (d,n,k)} f ~K(m,kn/f^2;d/f),
\end{align*}
as claimed.

Now apply the Weil bound $|K(a,b;q)| \leq \tau(q)q^{1/2}\mathrm{gcd}(a,b,q)^{1/2}$ to conclude
\begin{align*}
|S_n(m,k;d)|
	&\leq  \sum_{f \mid (d,n,k)} d^{1/2} f^{1/2} \tau\Big(\frac{d}{f}\Big) \mathrm{gcd}\Big(m,\frac{kn}{f^2},\frac{d}{f}\Big)^{1/2} \\
	&\ll_\epsilon d^{1/2+\epsilon} (d,k)^{1/2} (d,m)^{1/2},
\end{align*}
as desired.
\end{proof}

We first assume that $b > a + 3/2$.  We will shift the contour in \eqref{eqn:sab-hurwitz} to the line $\Re(s)=1-\delta$ for some $\delta > 1$.  Using Stirling's formula, along the line $\Re(s)=1-\delta$ for $\delta>1$, the integrand in \eqref{eqn:sab-hurwitz} is
\[
	\ll_{a,b,\delta} (1+|t|)^{a-b+2\delta-2}\sum_{d \geq 1} \frac{n^{b+1-\delta}}{d^{b-a+2-2\delta}} \sum_{k,m\geq 1} \frac{|S_n(m,k;d)|+|S_n(m,-k;d)|}{m^\delta k^{\delta + a}}. 
\]
The integral \eqref{eqn:sab-hurwitz} thus converges absolutely on the line $\Re(s)=1-\delta$ provided that $\delta < \frac{b-a+1}{2}$.  This is compatible with the assumption that $\delta>1$ by the assumption $b>a+3/2$.

Using Lemma \ref{lem:kloosterman}, 
the integral in \eqref{eqn:sab-hurwitz}, evaluated on the line $\Re(s)=1-\delta$, is
\[
	\ll_{a,b,\delta,\epsilon} \sum_{d \geq 1} \frac{n^{b+1-\delta}}{d^{b-a+\frac{3}{2}-2\delta-\epsilon}},
\]
by the assumption that $\delta > 1$.  Since $b>a+\frac{3}{2}$, we take $\delta = \frac{b-a}{2}+\frac{1}{4}-\epsilon$ and conclude the integral is
\[
	\ll_{a,b,\epsilon} n^{\frac{a+b}{2}+\frac{3}{4} + \epsilon} \sum_{d \geq 1} \frac{1}{d^{1+\epsilon}} \ll_{a,b,\epsilon} n^{\frac{a+b}{2}+\frac{3}{4} + \epsilon}.
\]
Together with the analysis of the poles, this yields the first part of Theorem \ref{thm:st-hurwitz}.

Now, assume that $b > \max\{a,2-a\}$.  Our goal in this case is to show that the contour in \eqref{eqn:sab-hurwitz} may be shifted to the line $\Re(s)=\sigma$ for some $0 < \sigma < 1$.  This is equivalent to obtaining sufficient cancellation in the series
\begin{equation} \label{eqn:afe-setup}
	\sum_{d \geq 1} d^{a-b-2s} \sum_{\substack{ 1 \leq e_1,e_2 \leq d \\ e_1e_2 \equiv n\pmod{d}}} \zeta(s-a, e_1/d) \zeta(s, e_2/d)
\end{equation}
on the line $\Re(s)=\sigma$.  We shall find it convenient to assume that $\sigma < a$ so that $\zeta(s-a,e_1/d)$ is related to an absolutely convergent Dirichlet series via the functional equation \eqref{eqn:hurwitz-fe}.  For $\zeta(s,e_2/d)$, we do not have this luxury, so we instead invoke the approximate functional equation of Lemma \ref{lem:AFE}.

In principle, in applying the functional equation for $\zeta(s-a,e_1/d)$ and the approximate functional equation for $\zeta(s,e_2/d)$, we are forced to consider six summations, corresponding to pairing each of the two terms in \eqref{eqn:hurwitz-fe} with the three terms in Lemma \ref{lem:AFE}.  However, the two summations in \eqref{eqn:hurwitz-fe} have the same shape as each other, as do the second and third summations in Lemma \ref{lem:AFE}.  Consequently, it essentially suffices to consider only two types of summation, corresponding to pairing the first term from Lemma \ref{lem:AFE} with a term from \eqref{eqn:hurwitz-fe} or pairing one of the latter two terms from Lemma \ref{lem:AFE} with a term from \eqref{eqn:hurwitz-fe}.

In the first of these two cases, where the first term of Lemma \ref{lem:AFE} for $\zeta(s,e_2/d)$ is paired with one of the terms in \eqref{eqn:hurwitz-fe} for $\zeta(s-a,e_1/s)$, we are led to consider series of the form
\begin{align}\label{eqn:afe-series-1}
\sum_d \frac{1}{d^{b-a+2s}} &\sum_{\substack{ 1 \leq e_1,e_2 \leq d \\ e_1e_2 \equiv n\pmod{d}}} \sum_{0 \leq k \leq T} \sum_{m \geq 1} \frac{e\big(\frac{me_1}{d}\big)}{(k+e_2/d)^s m^{1+a-s}} \\
	\notag & = \sum_d \frac{1}{d^{b-a+s}} \sum_{k \leq d(T+1)} \sum_{m \geq 1} \frac{1}{k^s m^{1+a-s}} \sum_{e_1 k \equiv n \pmod{d}} e\left( \frac{me_1}{d}\right),
\end{align}
where, as in Lemma \ref{lem:AFE}, we have set $T = \sqrt{2\pi(1+|t|)}$.  The exponential sum in \eqref{eqn:afe-series-1} is $0$ unless $(d,k) \mid (n,d,m)$, in which case it is of absolute value $(d,k)$.  Thus, since we have assumed $\Re(s)=\sigma < a$, \eqref{eqn:afe-series-1} is bounded by
\begin{align}\label{eqn:afe-bound-1}
\sum_{d\geq 1} \frac{1}{d^{b-a+\sigma}} \sum_{k \leq d(T+1)} \sum_{m \geq 1} \frac{(d,k)}{k^\sigma m^{1+a-\sigma}}
	&\ll \sum_{d \geq 1} \frac{1}{d^{b-a+\sigma}} \sum_{f \mid d} f^{1-\sigma} \left(\frac{Td}{f}\right)^{1-\sigma} \\
	\notag &\ll T^{1-\sigma} \sum_{d \geq 1} \frac{1}{d^{b-a+2\sigma-1-\epsilon}} \\
	\notag &\ll T^{1-\sigma} \\
	\notag &\ll (1+|t|)^{\frac{1-\sigma}{2}},
\end{align}
provided that 
$
	\sigma  > 1 - \frac{b-a}{2}.
$  
Since we have assumed $b>2-a$, there is some $\sigma < a$ for which this holds.  Using Stirling's formula, the additional factors in \eqref{eqn:hurwitz-fe} 
as applied to $\zeta(s - a, e_1/d)$
coming from the gamma function and exponentials may be bounded by $O((1+|t|)^{a-\sigma+\frac{1}{2}})$.  Altogether, the contribution to \eqref{eqn:afe-setup} from the first term in the approximate functional equation for $\zeta(s,e_2/d)$ is seen to be $O((1+|t|)^{a-\frac{3\sigma}{2}+1})$. 

We now consider the second type of summation, arising from the second and third terms in the approximate functional equation.  In particular, we are led to estimate
\begin{align}\label{eqn:afe-series-2}
\sum_d \frac{1}{d^{b-a+2s}}  \sum_{m \geq 1} \sum_{k \leq T} \frac{1}{k^{1-s} m^{a+1-s}} & \sum_{\substack{1 \leq e_1,e_2 \leq d \\ e_1e_2 \equiv n \pmod{d}}} e\left(\frac{\pm me_1\pm ke_2}{d}\right) \\
	\notag &=\sum_d \frac{1}{d^{b-a+2s}} \sum_{m \geq 1} \sum_{k \leq T} \frac{S_n(\pm m,\pm k;d)}{k^{1-s} m^{a+1-s}}.
\end{align}
We appeal to Lemma \ref{lem:kloosterman} to conclude that \eqref{eqn:afe-series-2} is bounded by
\begin{align}\label{eqn:afe-bound-2}
\sum_{d \geq 1} \frac{1}{d^{b-a+2\sigma}} \sum_{m \geq 1} \sum_{k \leq T} \frac{d^{1/2+\epsilon}(m,d)^{1/2}(k,d)^{1/2}}{k^{1-\sigma} m^{a+1-\sigma}}
	&\ll \sum_{d \geq 1} \frac{1}{d^{b-a+2\sigma-1/2-\epsilon}} \sum_{f \mid d} f^{\sigma-\frac{1}{2}} \left(\frac{T}{f}\right)^{\sigma} \\
	\notag & \ll T^\sigma \sum_{d \geq 1} \frac{1}{d^{b-a+2\sigma-1/2-\epsilon}} \\
	\notag & \ll T^\sigma \\
	\notag & \ll (1+|t|)^{\frac{\sigma}{2}}.
\end{align}
Once again, the additional factors in \eqref{eqn:hurwitz-fe} are of size $O((1+|t|)^{a-\sigma+\frac{1}{2}}$, while those in Lemma \ref{lem:AFE} are seen to be $O((1+|t|)^{\frac{1}{2}-\sigma})$.  We thus find that terms arising from the second and third summations in Lemma \ref{lem:AFE} contribute an amount that is $O((1+|t|)^{a-\frac{3\sigma}{2}+1})$ to \eqref{eqn:afe-setup}, matching the contribution from those terms arising from the first summation in Lemma \ref{lem:AFE}.  The error terms in Lemma \ref{lem:AFE} contribute a smaller amount, and we conclude that on the line $\Re(s)=\sigma$,
\begin{equation} \label{eqn:afe-conclusion}
	\sum_{d \geq 1} d^{a-b-2s} \sum_{\substack{ 1 \leq e_1,e_2 \leq d \\ e_1e_2 \equiv n\pmod{d}}} \zeta(s-a, e_1/d) \zeta(s, e_2/d)
		\ll (1+|t|)^{a - \frac{3\sigma}{2}+1},
\end{equation}
provided that $1 - \frac{b-a}{2} < \sigma < a$.

Thus, estimating the quotient of gamma factors by
Lemma \ref{lem:gamma_quotient},  the integrand in \eqref{eqn:sab-hurwitz} is $O_{a,b,\sigma}(n^{b+\sigma} (1+|t|)^{a-b-\frac{3\sigma}{2}})$.  The integral therefore converges absolutely on the line 
$\Re(s) = 1 - \frac{b-a}{2}+\epsilon$ for any $\epsilon>0$. This yields the second part of the theorem when
$\max\{a,2-a\} < b\leq  a + 3/2$.

\subsection{A simpler version of the error analysis}\label{sec:easy}
We present an alternative treatment of the error that avoids the complications
of the last section, obtaining a weaker error term of 
$o(n^{b + 1})$ for some ranges of the parameters. In particular, we assume that $a > 1$ and $b>a+2$.

Shift the contour in \eqref{eqn:sab-hurwitz} to 
$\Re(s) = 1 - \epsilon$ for small $\epsilon > 0$. We have $\zeta(s - a, e_1/d) \ll (1 + |t|)^{a - \frac{1}{2} + \epsilon}$
by the functional equation and Stirling's formula; we have 
$\zeta(s, e_2/d) \ll (1 + |t|)^{\epsilon} \cdot \big( \frac{e_2}{d} \big)^{-1}$ by the convexity bound, with the
term $\big( \frac{e_2}{d} \big)^{-1}$ arising from the first term $(e_2/d)^{-s}$ of 
 $\zeta(s, e_2/d)$; and 
we again use Lemma \ref{lem:gamma_quotient} to estimate the quotient of gamma functions. 

We conclude that the integrand is
\[
	\ll \sum_{d \geq 1} \frac{n^{b+1-\epsilon}}{d^{b-a-1-2\epsilon}} (1+|t|)^{a-b-\frac{3}{2}+2\epsilon}.
\]
This yields an error term of $O(n^{b + 1 - \epsilon})$ provided that the sum over $d$ and the integral over $t$ converge.  These conditions are satisfied for some $\epsilon>0$ if $b - a > 2$.

\section{Possible improvements}

As made clear in the discussion surrounding Lemma \ref{lem:kloosterman}, the error term in Theorem \ref{thm:st2} is controlled by sums of {Kloosterman sums} $K(r,s;q)$, where $q$ denotes the modulus.  The Weil bound implies that $K(r,s;q) \ll q^{1/2+\epsilon}$, and this is a key ingredient in the proof.  However, it is expected that much greater cancellation holds on average.  We expect that if the estimate $K(r,s;q) \ll q^{\theta+\epsilon}$ holds on average for some $0 \leq \theta \leq 1/2$, then the error term in Theorem \ref{thm:st2} may be improved to $O(n^{\frac{a+b}{2}+\frac{1+\theta}{2}+\epsilon})$.  Assuming a conjecture of Selberg \cite{Selberg}, the value $\theta=0$ is likely admissible, and this would yield a Ramanujan--Deligne quality error term in Theorem \ref{thm:st2}.  Using work of Deshouillers and Iwaniec \cite{DeshouillersIwaniec} on sums of Kloosterman sums, we speculate it may be possible to improve the error in Theorem \ref{thm:st2}, perhaps to the level $O(n^{\frac{a+b}{2}+\frac{7}{12}+\epsilon})$.
Alternatively, Shparlinski suggested to us that his work with Zhang \cite{SZ} on cancellation amongst Kloosterman sums to prime moduli could be readily generalized to the composite case
without difficulty, again leading to possible improvements. We leave these questions for future work.

Finally, as P.~Humphries pointed out to us, these questions can also be addressed via the spectral theory of automorphic forms. We refer to Kuznetsov \cite{kuznetsov} and Motohashi \cite{motohashi}
for some related results along these lines, including a treatment by Motohashi of the case $a = b = 0$. Humphries suggested to us that these techniques may be able to address complex $a$ and $b$
in greater generality, and again we leave this question for future work.

\bibliographystyle{abbrv}
\bibliography{references}

\end{document}